\documentclass[11pt,notitlepage,twoside,a4paper]{amsart}
\usepackage{tabularx,booktabs}
\usepackage{geometry}
\usepackage{caption}
\usepackage{amsmath}
\usepackage{amsfonts}

\usepackage{amscd}
\usepackage{amsthm}
\usepackage{amssymb} \usepackage{latexsym}
\usepackage{eufrak}
\usepackage{euscript}
\usepackage{epsfig}
\usepackage{graphics}
\usepackage{array}
\usepackage{enumerate}
\usepackage{color}
\usepackage{wasysym}
\usepackage{hyperref}
\usepackage{pdfsync}

\usepackage{graphicx}
\usepackage{float}
\usepackage{subfigure}

\newcommand{\ba}{\begin{eqnarray}}
\newcommand{\ea}{\end{eqnarray}}

\newcommand{\qe}{\end{equation}}

\newcommand{\M}{\mathcal{M}_g}
\newcommand{\HH}{H^1_{-1}(\hat{X},\mathbb{C})}
\newcommand{\PS}{\mathcal{P}_g}
\newcommand{\NRH}{H^1_{-1}(\hat X,\Sigma,\mathbb{C})}
\newcommand{\NH}{H^1_{-1}(\hat X,\mathbb{C})}

\newtheorem{theorem}{Theorem}[section]

\newtheorem{remark}[theorem]{Remark}

\newtheorem{proposition}[theorem]{Proposition}

\begin{document}
\title[Avila-Gou\"ezel-Yoccoz and Teichm\"uller]{A comparison between Avila-Gou\"ezel-Yoccoz norm and  Teichm\"uller norm}

\author{Weixu Su}
\author{Shenxing Zhang}

\address{Weixu Su: School of Mathematics, Sun Yat-sen University, Guangzhou 510275, China}
\email{\href{mailto:suwx9@mail.sysu.edu.cn}
{suwx9@mail.sysu.edu.cn}}
\address{Shenxing, Zhang: School of Mathematical Sciences, Fudan University, Shanghai 200433, China}
\email{\href{mailto:21110180030@m.fudan.edu.cn}
{21110180030@m.fudan.edu.cn}}

\begin{abstract}
We give a comparison between the Avila-Gou\"ezel-Yoccoz
norm and the Teichm\"uller norm on the principal stratum of
holomorphic quadratic differentials.

 \medskip


\medskip

\end{abstract}
\maketitle

\section{Introduction}

Let $X$ be a compact Riemann surface of genus $g$.
A holomorphic quadratic differential $q$ on $X$ is a tensor given locally by an expression $q=q(z) d z^{2}$ where $z$ is a conformal coordinate on $X$ and $q(z)$ is holomorphic. Such a (nonzero) quadratic differential $q$  defines a flat metric $|q|^{1/2}$ on $X$. This metric has conical singularities at zeroes of $q$. Its area
is defined by
$$\|q\|=\int_{X} |q(z)| |dz|^2 .$$

Fix $g\geq 2$ and let $\PS$ be the \emph{principal
stratum} of the moduli space of quadratic differentials,
consisting of isomorphism classes of holomorphic quadratic
differentials $(X,q)$ with $4g-4$ distinct simple zeroes.

There is a Finsler metric on $\PS$ called \emph{AGY metric}, which was introduced by Avila-Gou\"ezel-Yoccoz \cite[\S 2.2.2]{AGY}.
This norm plays an important role in the study of Teichm\"uller flow. See \cite{AG,AGY,EM}.

Let $\M$ be the moduli space of Riemann surfaces
of genus $g$. Let $\pi: \PS \to \M$ be the natural projection,  defined by $\pi(X,q)=X$.
In the note, we consider the derivative of $\pi$ and compare the AGY
norm  with the Teichm\"uller norm.

Recently, Kahn-Wright \cite{KW21} derived  a comparison between  the Hodge norm (another important norm on $\PS$) and the Teichm\"uller norm.
Our research is motivated by their work.

\bigskip

For each $(X,q)\in \PS$,
there is a \emph{canonical double cover} $\rho: \hat X \to X$, ramified at the odd zeros of $q$,
such that $\rho^*q$ is the square of an Abelian differential $\omega$ on $\hat X$.
See \cite{DH} or \cite[\S 2]{L} for details.
The Abelian differential $\omega$ is a
$-1$ eigenvector for the holomorphic involution
$\tau: \hat X \to \hat X$ that permutes the sheets of the double cover,
that is, $$\tau^*\omega=-\omega.$$ We can
identify the tangent space of $\PS$ at
$(X,q)$ as $\HH$,  the $-1$ eigenspace
for the action of $\tau$ on the cohomology group $H^1(\hat X, \mathbb{C})$.

Every element of $H^1(\hat X, \mathbb{C})$ can be represented uniquely by a harmonic one-form. Consequently,
there is a natural decomposition of $\HH$ into
$H_{-1}^{1,0}(\hat{X})\oplus H_{-1}^{0,1}(\hat{X}).$
Note that the kernel of $D \pi$ is $H_{-1}^{1,0}(\hat{X})$. See
Theorem \ref{KW:diff} below.

We consider $\eta \in H_{-1}^{0,1}(\hat{X})$ and compare the AGY norm of $\eta$ with the Teichm\"uller norm of $D \pi(\eta)$.
The main result is

\begin{theorem}\label{thm:main}
Let $(X,q)\in \PS$ with area $\|q\|=1$.
 Let  $\rho: \hat X\to X$ be the canonical double cover
 such that $\rho^*q=\omega^2$.  Then for any $\eta \in H_{-1}^{0,1}(\hat{X})$,
we have
\begin{equation}\label{equ:main}
  \frac{r}{2\sqrt{2}} \ {\|\eta\|_{\mathrm{AGY}}}\leq {\|D \pi(\eta)\|_{\mathrm{Teich}}} \leq \frac{8}{\sqrt{\pi}r} \|\eta\|_{\mathrm{AGY}},
\end{equation}
where  $2 r$ is the shortest length of saddle connections on $(\hat{X}, \omega)$.
\end{theorem}

\begin{remark}
Note that the area of $\omega$ is $2$.
\end{remark}

The paper has the following structure.
In \S 2, we present some basic properties of quadratic differentials.
The upper bound in
\eqref{equ:main} is proved in \S 3,  where we use the Delaunay triangulation
of quadratic differential
to construct quasiconformal maps with explicit Beltrami differentials.
In \S 4,
we give an upper bound of the AGY norm in terms of the Hodge norm,
 and then we derive the lower bound in
\eqref{equ:main} from Kahn-Wright\cite[Theorem 1.4] {KW21}.

\section{Preliminaries}

\subsection{The moduli space of quadratic differentials}

Let $g \geq 2$. We denote by $\M$ the moduli space of compact Riemann surfaces of genus $g$.
For $X\in \M$, the cotangent space of $\M$ at $X$ is canonically identified with the space
$Q(X)$ of holomorphic quadratic differentials  on $X$.
We define the $L^1$-norm on $Q(X)$ by
$$\|q\|=\int_X |q| .$$

A tangent vector of $\M$ at $X$ is represented by a Beltrami differential $\mu$. There is a natural pairing between quadratic differentials and Beltrami differentials given by
$$\langle \mu, q\rangle= \int_X \mu  q .$$

The \emph{Teichm\"uller norm} of $\mu$ is defined by
$$\|\mu\|_{\mathrm{Teich}}=\sup_{\|q\|=1} \operatorname{Re} \  \langle \mu, q\rangle.$$
This gives the infinitesimal form of the Teichm\"uller metric
on $\M$.

\medskip

Let $\mathcal{Q}_g$ be the moduli space of  quadratic differentials,
consisting of pairs $(X,q)$
where $X$ is a compact Riemann surface of genus $g$ and $q$ is a holomorphic quadratic
differential on $X$. The moduli space
$\mathcal{Q}_g$ has a stratified structure:
given a positive integral vector
$\kappa=\left(\kappa_1,\cdots,\kappa_n\right)$ with $\sum \kappa_i=4g-4$,
we let $\mathcal{Q}_g(\kappa)\subset \mathcal{Q}_g$
be the set of quadratic differentials $(X,q)$ where $q$
has $n$ zeros of order $\kappa_1,\cdots, \kappa_n$.

In the paper, our study is mainly restricted on the principal stratum,
consisting of those quadratic
differentials all of whose zeros are simple. We denote the principal stratum by $\PS$.
This stratum is both open and dense in $\mathcal{Q}_g$.

\subsection{Canonical double cover}

Let $\mathcal{Q}_g(\kappa)$ be a stratum of quadratic differentials.
Given $(X,q)\in\mathcal{Q}_g(\kappa)$, let $\rho: \hat X \to X$
be the canonical double cover such that the pull-back
$\rho^*q$ becomes the square of an Abelian differential $\omega$ on $\hat X$.
Let $\tau: \hat X \to \hat X$ be the  involution
that permutes the sheets of the double cover.
By the construction, $\tau^*\omega=-\omega$.

Let $\Sigma$ be the set of zeros of $\omega$.
 Denote by $H^1_{-1}(\hat X,\Sigma,\mathbb{C})$
the $-1$ eigenspace for the action of $\tau$
on  the relative homology
group $H^1(\hat X,\Sigma,\mathbb{C})$.
Note that the relative cohomology class of $\omega$
 is an element of $H^1_{-1}(\hat X,\Sigma,\mathbb{C})$.
A neighborhood of $\omega$ in $\NRH$
gives a local chart of $q$ in the stratum,
via the period mapping.

In the following, we shall identify the tangent space at $(X,q)$
with the cohomology $\NRH$. If $(X,q)\in \PS$,
then $q$ has no zeros of even order. In this case,
since $\Sigma$ is the set of fixed points of $\tau$,
we have
$$\NRH \cong \NH.$$
Thus each element of $\NH$ can be uniquely represented by a harmonic $1$-form.

\medskip

The following result describes the tangent map of
$\pi: \PS \to \M$ in terms of the period coordinates.
It is proved by Kahn-Wright \cite[Corollary 1.2]{KW21}.

\begin{theorem}\label{KW:diff}
Consider the projection $\pi: \PS \to \M$. Let $(X,q)\in \PS$ and let $\eta$ be a harmonic $1$-form on $\hat X$ that represents an element of $\NH$.
Then for any $\phi\in Q(X)$, we have

$$\langle D\pi(\eta), \phi\rangle=\frac{1}{2}\int_{\hat X} \rho^*(\phi)\frac{\eta^{0,1}}{\omega},$$
where $\eta^{0,1}$ is the anti-holomorphic part of $\eta$.

\end{theorem}

\subsection{The AGY norm.}

The AGY norm is defined by Avila-Gou\"ezel-Yoccoz \cite{AGY} on any stratum of Abelian differentials.

With the notations in \S 2.2, we consider the Abelian differential $\omega$ as an element of $H^1(\hat X,\Sigma,\mathbb{C})$. A saddle connection of $\omega$ is a geodesic segment for the flat metric defined by $|\omega|$
joining two zeros of $\omega$ and not passing any zero in its interior.
Each saddle connection $\gamma$ gives rise to an element $[\gamma]$ of
the homology $H_1(\hat X,\Sigma,\mathbb{C})$. And the set of saddle connections generates the the homology $H_1(\hat X,\Sigma,\mathbb{C})$.
Denote by $\left\{\gamma_{j}\right\}$ the set of saddle connections on $\omega$.

For any $[\eta]\in H^1(\hat X,\Sigma,\mathbb{C})$,
its AGY norm is defined by
\begin{equation*}
    \|\eta\|_{\mathrm{AGY}}=\sup_{\gamma_j}  \frac{\left|\int_{\gamma_{j}} \eta\right|}{\left|\int_{\gamma_{j}} \omega\right|} ,
\end{equation*}
where the supremum is taken over all saddle connections.

Avila-Gou\"ezel-Yoccoz \cite[\S 2.2.2]{AGY} showed that the AGY norm is continuous and induces
a complete metric on each stratum.

\section{The upper bound}
In this section, we give an upper bound of
$\|D\pi(\eta)\|_{\mathrm{Teich}}$
in terms of
$\|\eta\|_{\mathrm{AGY}},$
for any $\eta\in H^1_{-1}(\hat X, \mathbb{C})$.
The idea is to triangulate the surface and
compute the Beltrami differentials of maps that are affine on each triangle.
We remark that the proof  applies to any other stratum of quadratic differentials or Abelian differentials.

\subsection{Delaunay triangulation}
Given a quadratic differential $(X,q)$,
there is an associated flat metric (with conical singularities)
on $X$, defined by $|q|^{1/2}$.
Denote by $\Sigma$ the set of zeros of $q$.
For any $x\in X$,
let $d(x, \Sigma)$ be the minimal $|q|^{1/2}$-distance from $x$ to $\Sigma$.
%We define the \emph{diameter} of $(X,q)$ by

%  $$\operatorname{diam}(X,q)=\max_{x} d(x,\Sigma).$$

The next result is proved by Masur-Smillie  \cite[\S 4]{MS}.
See also Farb-Masur \cite[Proposition 3.1]{FM}.

\begin{proposition}
Let $(X,q)$ be a holomorphic quadratic differential of area $\|q\|\leq 1$. There is a triangulation $\Delta$ on $X$ with the following properties:
\begin{enumerate}
  \item The vertices of $\Delta$ lie in the zero set of $q$.
  \item The edges of $\Delta$ are saddle connections of $q$.
  \item Each triangle is inscribed in a circle of radius $d(x, \Sigma)$ for some $x\in X$.
\end{enumerate}
\end{proposition}
The above construction is called a \emph{Delaunay triangulation} of $q$.

Let $s=\sqrt{\frac{2}{\pi}}$, and let $B_s$ be the set of points in $X$
with $d(x,\Sigma)\leq s$.  By  the proof of \cite[Theorem 5.3]{MS},
the complement of $B_s$ is contained in a union of disjoint maximal flat cylinders,
with the property that their circumference is less than their height.

\subsection{The proof of upper bound}

Let $\eta\in H^1_{-1}(\hat X, \mathbb{C})$.
Denote by $(\hat X_t, \omega_t)$ the family of
Abelian differentials corresponding to the cohomology
classes
$\omega+t \eta\in H^1_{-1}(\hat X, \mathbb{C})$,
for sufficiently small $t>0$.

Let $\Delta$ be a Delaunay triangulation of
$(\hat X, \omega)$. By the construction,
the vertices of $\Delta$ are the zeros of $\omega$, and the edges of $\Delta$ are saddle connections of $\omega$. For each $t$,
we can straighten $\Delta$ to be a triangulation
of $\hat X_t$ (not necessary Delaunay), denoted by $\Delta_t$, such that
the edges  are saddle connections of $\omega_t$.

The next step is to construct quasiconformal mappings
$f_t$ from $\hat X$ to $\hat X_t$ that are linear
on each triangle.
 Denote the Beltrami differentials of $f_t$ by
 $\mu_t$. Then
 $$D\pi(\eta)\cong\frac{d \mu_t}{dt}|_{t=0}.$$

\begin{proposition}\label{prop:upper}
Let  $2r$ be the shortest length of saddle connections on $(\hat{X}, \omega)$. Then
$$\left\| D\pi(\eta) \right\|_{\mathrm{Teich}} \leq \frac{8}{\sqrt{\pi}r} \|\eta\|_{\mathrm{AGY}} .$$
\end{proposition}
\begin{proof}
Denote by $$\mu=\frac{d \mu_t}{dt}|_{t=0}.$$ Since $\|\mu\|_{\mathrm{Teich}}\leq \|\mu\|_\infty$,
it suffices to give the upper bound for $\|\mu\|_\infty$.

Let $T=\triangle OAB$ be any triangle of $\Delta$,
where $O,A,B$ denotes the vertices.
For simplicity, we consider $T$ as a triangle in the complex plane and put
$O=0, A=a>0$ and $B=b\in \mathbb{C}$.
By definition,
$$a=\int_{\gamma} \omega, b=\int_{\gamma'}\omega,$$
where $\gamma$ and $\gamma'$ denote the saddle connection connecting $O$ to $A$ and $O$ to $B$,
respectively.

\begin{figure}[htbp]
\centering
\includegraphics[width=10cm]{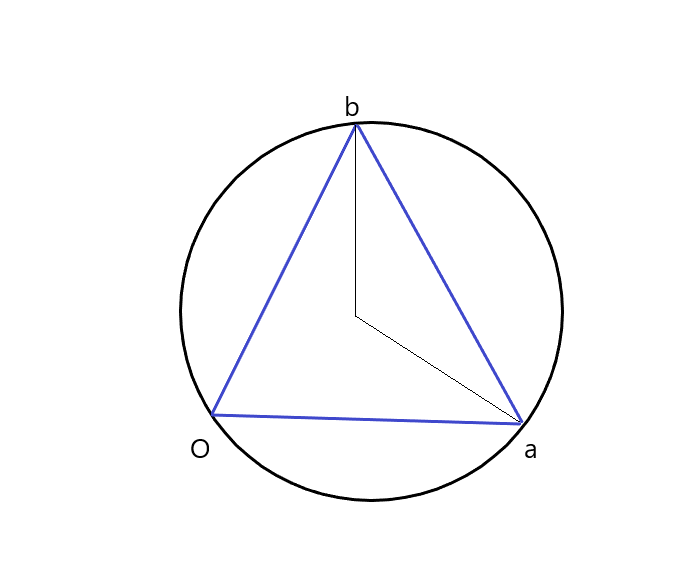}
\end{figure}

For each $t$ sufficiently small, the corresponding triangle in $\hat X_t$ has vertices given by $0$, $a + t \alpha$ and $b + t \beta$, where $$\alpha = \int_{\gamma} \eta,\beta = \int_{\gamma'} \eta.$$

Denote the associated affine mapping between the triangles by $$f_t(z) = Rz + S\bar{z}.$$

Then we have
\[
Ra+Sa = a +t\alpha,
\]
\[
Rb + S\bar b = b + t \beta.
\]

A simple computation shows that the Beltrami coefficient $\mu_t$ is equal to
\[
\frac{S}{R} = t\frac{\frac{\alpha}{a}-\frac{\beta}{b}}{1-\frac{\bar b}{b}}+o(t).
\]

Now we give an upper bound of
$$|\mu(z)|=\left|\frac{\frac{\alpha}{a}-\frac{\beta}{b}}{1-\frac{\bar b}{b}} \right|.$$

Assume that $\theta=\arg b$. Then
$$|1-\frac{\bar b}{b}|=2|\sin\theta|.$$

To give an upper bound of the quasiconformal dilatation, we discuss $\sin\theta$ in two cases.

Let $s_0=\sqrt{\frac{4}{\pi}}=\frac{2}{\sqrt{\pi}}$. We remark that the area of $|\omega|$ is $2$.
Note that for any edge of $T$, it either has length $\leq 2s_0$ or crosses a maximal flat cylinder
$C$ whose height $h$ is greater than its circumference $c$.

Assume that all edges of $T$ has length $\leq 2s_0$. In this case,  the triangle $T$ is inscribed in a circle of radius $d(x, \Sigma) \leq 2s_0.$

Since $\sin\theta=|a-b|/2 d(x,\Sigma)$, we have
$$|\sin\theta|\geq \frac{r}{d(x,\Sigma)}\geq \frac{\sqrt{\pi}r}{4}.$$
Thus we have
$$\left|\frac{\frac{\alpha}{a}-\frac{\beta}{b}}{1-\frac{\bar b}{b}} \right|\leq \frac{ 8 \max\{|\frac{\alpha}{a}|, |\frac{\beta}{b}|\}}{\sqrt{\pi}r}\leq \frac{8}{\sqrt{\pi}r} \|\eta\|_{\mathrm{AGY}}.$$

The remaining case is that some edge of $T$ crosses a maximal flat cylinder
$C$ whose height $h$ is greater than its circumference $c$.
In this case, some other edge of $T$ also crosses $C$.
Thus the triangle $T$ looks like an isosceles triangle with a short base.
As a result, we may choose the angle $\theta$ such that
$$\frac{\pi}{4}\leq \theta \leq \frac{\pi}{2}.$$
Then we have $\sin\theta \geq \frac{\sqrt{2}}{2}$.
It follows that
$$\left|\frac{\frac{\alpha}{a}-\frac{\beta}{b}}{1-\frac{\bar b}{b}} \right|\leq \frac{ 2 \max\{|\frac{\alpha}{a}|, |\frac{\beta}{b}|\}}{\sqrt{2}}\leq \sqrt{2} \|\eta\|_{\mathrm{AGY}}.$$

Note that $\pi r^2\leq 2$ and then $\sqrt{2}\leq \frac{2}{\sqrt{\pi}r}$.
This completes the proof. 

\end{proof}

\begin{remark}
It is known that for any quadratic differential $q$, in the the direction
of Teichm\"uller flow, the AGY norm is less than the Teichm\"uller norm
(see \cite[Page 152]{AGY}).

As we have shown in the proof of Proposition \ref{prop:upper},
the order $\frac{1}{r}$ appears when the triangle is almost flat.
If there is some angle of the triangle which is neither close to $0$ or $\pi$,
then the Beltrami coefficient should be bounded above by
$ \|\eta\|_{\mathrm{AGY}}$ up to a multiplicative constant.
\end{remark}

\subsection{The order $\frac{1}{r}$ in Proposition \ref{prop:upper} is sharp.}

We recall the following construction of Kahn-Wright \cite[\S 3.3]{KW21}. 

Let $\epsilon>0$ be a small constant. 
We take a square torus of length $1$ and make a length
$\epsilon$ horizontal slit. Then we identify the endpoints of the slit
to make a figure-eight and glue in a  cylinder with circumference $\epsilon$
and height $\epsilon$. The construction defines an Abelian differential
$(X_\epsilon,\omega_\epsilon)$ with one double
zero, i.e. a translation surface in $\mathcal{H}(2)$. 

\begin{figure}[htbp]
\centering
\includegraphics[width=8cm]{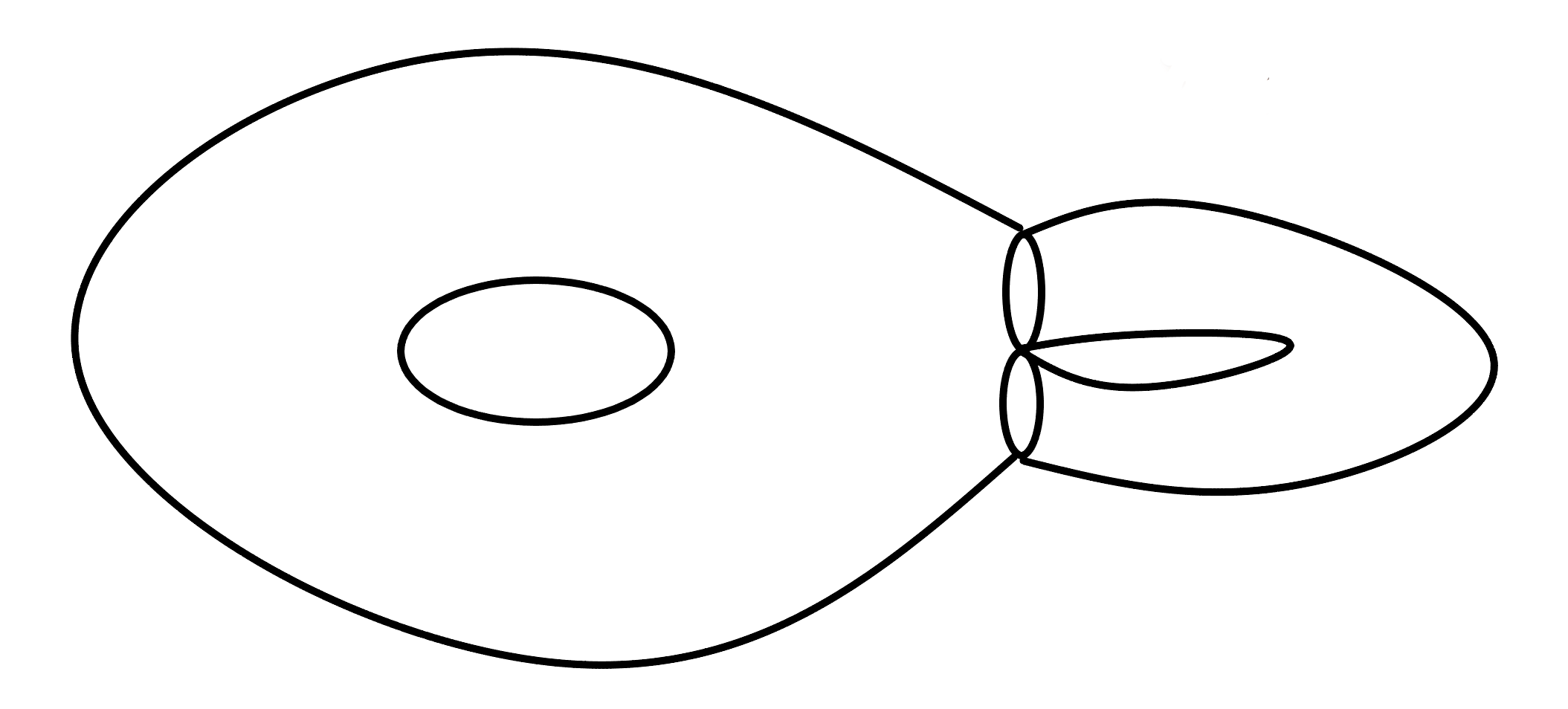}
\caption{The examples of Kahn-Wright \cite[\S 3.3]{KW21}.}
\end{figure}

Let ${\gamma}_{\epsilon}$ be the core curve of the small cylinder in $(X_{\epsilon},{\omega}_{\epsilon})$.
Denote the harmonic differential dual to $\gamma_{\epsilon}$ by ${\gamma}_{\epsilon}^*$.  

\begin{remark}
We can write $\gamma_\epsilon^*=\beta_\epsilon+\bar\beta_\epsilon$,
where $\beta_\epsilon$ is an Abelian differential. It is known that the Hodge norm
of $\beta_\epsilon$ is bounded above and below independently
of $\epsilon$.
\end{remark}

As shown by Kahn-Wright \cite[\S 3.3]{KW21}, 
$$\left\| D\pi({\gamma}_{\epsilon}^*) \right\|_{\mathrm{Teich}} \geq \frac{C}{\epsilon}$$
for some constant $C$.

The path $\omega_\epsilon+t \epsilon \gamma_\epsilon^*$ 
is corresponding to a family of translation surfaces, 
obtained by twisting along the core curve of the small cylinder.
When $t=1$, $\omega_\epsilon+ \epsilon \gamma_\epsilon^*$ is a Dehn twist of $\gamma_\epsilon$.

The length of shortest saddle connections of $\omega_\epsilon$ is equal to $\epsilon$.
If $\alpha_0$ is the shortest saddle connection contained in the small cylinder and crossing $\gamma_\epsilon$,
then 
$$\frac{\left|\int_{\alpha_0} \gamma_\epsilon^*\right|}{\left|\int_{\alpha_0} \omega_\epsilon \right|}=\frac{\epsilon}{\epsilon}=1.$$

For any other saddle connection $\alpha$, which crosses the small cylinder
 $n$  times, we have

 $$\frac{\left|\int_{\alpha} \gamma_\epsilon^*\right|}{\left|\int_{\alpha} \omega_\epsilon \right|}\leq \frac{n \epsilon}{ n\epsilon}=1.$$

As a result,  $\|{\gamma}_{\epsilon}^*\|_{\mathrm{AGY}}= 1$. In conclusion, we have

$$\left\| D\pi({\gamma}_{\epsilon}^*) \right\|_{\mathrm{Teich}} \geq C \frac{\|{\gamma}_{\epsilon}^*\|_{\mathrm{AGY}}}{\epsilon},$$
where $\epsilon$ is the length of shortest saddle connections of $\omega_\epsilon$.

\section{The lower bound}
In this section, we consider tangent vectors to $\PS$ of the form $\eta=\bar{\beta}$, where $\beta \in H_{-1}^{1,0}(\hat{X})$. By Theorem \ref{KW:diff}, the Beltrami differential $\mu=\bar{\beta} / \omega$ can be considered as the tangent vector $D \pi(\eta)$ via the pairing with holomorphic quadratic differentials
$$\int_{\hat{X}} \rho^*(\phi)\frac{ \bar{\beta}}{\omega}.$$

The Hodge norm of $\beta \in H_{-1}^{1,0}(\hat{X})$ is defined by
$$
\|\beta\|_{\mathrm{Hodge}}=\sqrt{\int_{\hat{X}}|\beta|^{2}} .
$$

We have (see \cite[Theorem 3.1]{KW21}):
\begin{theorem}\label{thm:KW}
For any $\eta=\bar{\beta} \in H_{-1}^{0,1}(\hat{X})$, we have
$$
\|D \pi(\eta)\|_{\mathrm{Teich}} \geq \frac{\|\beta\|_{\mathrm{Hodge}}}{\|\omega\|_{\mathrm{Hodge}}} \ .
$$
\end{theorem}

\begin{proof}[The lower bound in Theorem \ref{thm:main}]
	
The assumption  $\|q\|=1$ implies $\|\omega\|_{\mathrm{Hodge}}=\sqrt{2}.$
	Applying Theorem \ref{thm:KW}
	and the next proposition , we have
	
	\begin{eqnarray}\label{equ:lower}
		\|D \pi(\eta)\|_{\mathrm{Teich}} &\geq& \frac{\|\eta\|_{\mathrm{Hodge}}}{\sqrt{2}}
		\geq \frac{r}{2\sqrt{2}} \|\eta\|_{\mathrm{AGY}} .
	\end{eqnarray}

\end{proof}

\begin{proposition}\label{prop:lower}
Let $2r$ be the shortest length of  saddle connections. For any saddle connection $\gamma$ of $\omega$
and any  $\beta\in H_{-1}^{1,0}(\hat{X})$,
we have
$$\frac{\left| \int_\gamma \beta \right|}{\left| \int_\gamma \omega \right|} \leq 2 \frac{\|\beta\|_{\mathrm{Hodge}}}{r}.$$
As a result, for any $\eta=\bar{\beta} \in H_{-1}^{0,1}(\hat{X})$, we have
$$\|\eta\|_{\mathrm{AGY}} \leq 2 \frac{\|\eta\|_{\mathrm{Hodge}}}{r}.$$
\end{proposition}

\begin{proof} 
	We shall endow the surface with the metric defined by $|\omega|$. 
Let $\Sigma$ be the set of zeros of $\omega$.
Given a  saddle connection $\gamma$ of $\omega$, we can decompose $\gamma$ into two parts. Either a segment of $\gamma$ is contained in the disk $D(p,r)$ of radius $r$ centered at a $p\in \Sigma$ and such a segment intersects with
$D(p, r/2)$; or the segment is outside $D(p,r/2)$ for all $p\in \Sigma$.
We denote the two parts by $\gamma'$ and $\gamma''$.

It is not necessary that $\gamma'$ or $\gamma''$ is connected. We write
$$\gamma'=\bigcup_i \gamma_i' \ \mathrm{and} \ \gamma''=\bigcup_j \gamma_j'',$$
where $\gamma'_i ,\gamma_j''$ denote the connected components. 

For each $\gamma_i'$, there is a unique zero (of order $2$) $z_i$ of $\omega$ such that 
$\gamma_i$ is contained in the disk $D(z_i,r)$ and $\gamma_i$ intersects with $D(z_i,r/2)$.  It follows from  \cite[Lemma 3.2]{KW21} that, for $z$ in $D(z_0,r)$, 
$$\left| \int_{z_i}^z \beta \right| \leq \|\beta\|_{\mathrm{Hodge}}. $$
As  a result, 
$$\left|\int_{\gamma_i'}\beta\right| \leq 2 \|\beta\|_{\mathrm{Hodge}}.$$
Since $\gamma$ crosses the annulus $D(z_i ,r) \setminus D(z_i, r/2)$, 
$\left|\int_{\gamma_i'}\omega\right| \geq r$. This implies 

\begin{equation}\label{equ:zero}
	\frac{\left|\int_{\gamma_i'}\beta\right|}{\left|\int_{\gamma_i'}\omega\right|} \leq 2 \frac{\|\beta\|_{\mathrm{Hodge}}}{r}.
\end{equation}

Now we consider $\gamma_j''$. We have
\begin{eqnarray*}
\left|\int_{\gamma_j''} \beta\right| \leq \int_{\gamma_j''} \left| \frac{\beta}{\omega}\right| \left|\omega\right|.
\end{eqnarray*}
We give an upper bound for $\frac{\beta}{\omega}$. Let $x_0\in \gamma_j''$. 
Note that there is a disk $D(x_0, r/2)$ of radius $r/2$ around $x_0$,
which does not contain any zeros of $\omega$.

Let $z$ be the natural coordinate of $\omega$ on $D(z_0, r/2)$, where $z_0=z(x_0)$.
In $D(z_0, r/2)$, we have $\omega=d z$. And $\frac{\beta}{\omega}(z)$ defines a holomorphic function on $D(z_0, r/2)$.

By the mean-value inequality of subharmonic function, we have 
\begin{eqnarray*}
\left|\frac{\beta}{\omega}(z_0) \right|^2 &\leq& \frac{4 \int_{D(z_0, r/2)}\left|\frac{\beta}{\omega}(z) \right|^2 dxdy}{\pi r^2}.\\
\end{eqnarray*}

Thus  we have
\begin{eqnarray*}
	\left|\frac{\beta}{\omega}(z_0) \right| &\leq& \frac{2 \left(\int_{D(z_0, r/2)}\left|\frac{\beta}{\omega}(z) \right|^2 dxdy\right)^{1/2}}{\sqrt{\pi} r}\\
	&=& \frac{2 \left(\int_{D(z_0, r/2)}\left|\beta \right|^2 \right)^{1/2}}{\sqrt{\pi} r}\\
&\leq& \frac{2 \|\beta\|_{\mathrm{Hodge}}}{\sqrt \pi r}.
\end{eqnarray*}

As a result, we show

\begin{equation}\label{equ:one}
\frac{\left|\int_{\gamma''} \beta\right|}{\left|\int_{\gamma''} \omega \right|} \leq \max_{\gamma''} \left| \frac{\beta}{\omega}\right| 
\leq  2 \frac{\|\beta\|_{\mathrm{Hodge}}}{\sqrt{\pi} r},
\end{equation}

Combining \eqref{equ:one} with \eqref{equ:zero}, we have

\begin{eqnarray*}
\frac{\left|\int_{\gamma} \beta\right|}{\left|\int_{\gamma} \omega \right|}
&\leq&  \frac{\sum_i \left|\int_{\gamma_i'} \beta\right|+ \sum_j \left|\int_{\gamma_j''} \beta\right|}{\sum_i \left|\int_{\gamma_i'} \omega\right|+ \sum_j
\left|\int_{\gamma_j''} \omega\right|} \\
&\leq& \max_{i,j}\left\{ \frac{\left|\int_{\gamma_i'} \beta\right|}{\left|\int_{\gamma_i'} \omega\right|}, \frac{\left|\int_{\gamma_j''} \beta\right|}{\left|\int_{\gamma_j''} \omega\right|}\right\} \\
&\leq & 2 \max\left\{\frac{ \|\beta\|_{\mathrm{Hodge}}}{r}, \frac{\|\beta\|_{\mathrm{Hodge}}}{\sqrt{\pi} r}\right\} \\
&=& 2 \frac{ \|\beta\|_{\mathrm{Hodge}}}{r} .
\end{eqnarray*}

\end{proof}

%\section{Two examples}

%\subsection{The upper bound.}

%Let $T$ be a flat cylinder with horizontal closed trajectories of circumstance $\epsilon$ and height $\frac{1}{\epsilon}$.
%We can glue $T$ into a compact Riemann surface $X$ of genus $2$ and obtain a quadratic differential $q$ in the principal stratum.
%We assume that the short length of saddle connections of $q$ is comparable to $\epsilon$.

%Let $C$ be a core geodesic of $T$.
%Consider a family of deformations of $q$ obtained by
%cutting $T$ along $C$, twisting along $C$ with length $s$ and gluing back.
%Denote the  tangent vector to $\PS$ corresponding to $\frac{\partial}{\partial s}|_{s=0}$ by $\eta$, and let $\mu=D\pi (\eta)$.

%For any saddle connection $\gamma$ on $X$ not contained in the boundary of $T$,
%its length is comparable to $\frac{i(\gamma, C)}{\epsilon}.$
%One can show that
%$$\left|\int_\gamma \eta\right| =O (1).$$
%Thus
%$$\|\eta\|_{\mathrm{AGY}} = O(\epsilon).$$

%On the other hand, we have

%$$\|\mu\|_{\mathrm{Teich}} \asymp \frac{1}{\epsilon}.$$

%This shows that the coefficient $\frac{1}{r^2}$ in Theorem \ref{thm:main} is sharp.

\bibliographystyle{plain}

\bibliography{bibliography}

\end{document}